\documentclass[10pt,doublespace]{article}
\usepackage{psfrag}
\usepackage{amssymb,bbm}
\usepackage{amsmath}
\usepackage{amsfonts}
\usepackage{pstricks,pstricks-add}
\usepackage{graphicx}
\usepackage[latin5]{inputenc}
\usepackage{yfonts}
\usepackage{setspace}
\usepackage[toc,page]{appendix}
\newcounter{theorem}

\newtheorem{lemma}{Lemma}

\newenvironment{proof}[1][Proof]{\textbf{#1.} }{\rule{0.5em}{0.5em}}

\title{Tube Volumes via Functional Equations}

\author{Ali Deniz, Şahin Koçak\footnote{Corresponding Author}, Yunus Özdemir, A. Ersin Üreyen}

\begin{document}

\maketitle

\begin{abstract}
We give a very simple approach to the computation of tube volumes of self-similar sprays using a functional equation satisfied by them.
\end{abstract}


\section{Introduction}\label{introduction}

Complex dimensions of fractals were introduced in 90's by M. Lapidus. He and his coworkers, especially M. von Frankenhuijen and E. Pearse established in the sequel an elaborate theory of complex dimensions of strings, sprays and various higher dimensional generalizations of them (see \cite{LaFra}, \cite{LPW} and references therein). In this note, we want to give a simple alternative approach for self-similar sprays.

A self-similar spray generated by an open set $G\subseteq \mathbb{R}^n$ is a collection $(G_k)_{k\in \mathbb{N}}$ of pairwise disjoint sets $G_k\subseteq \mathbb{R}^n$ such that $G_k$ is a scaled copy of $G$ by some $\lambda_k>0$, where the sequence $(\lambda_k)_{k\in \mathbb{N}}$, called the associated scaling sequence of the spray, is obtained from a ratio list $\{r_1, r_2, \dots, r_J\} \, (0<r_j<1)$ by building all possible words of multiples of the ratios $r_j$:
\[1, r_1, r_2, \dots, r_J, r_1 r_1, r_1r_2, \dots, r_1r_J, r_2r_1, r_2 r_2, \dots,  r_2r_J, \dots\]

For an example see Figure~\ref{figurespraytube} where the ratio list consists of $\frac{1}{2}$, $\frac{1}{3}$ and $\frac{1}{4}$.

We recall that the unique real number $D$ satisfying the Moran equation \linebreak \mbox{$r_1^D+\cdots+ r_J^D=1$} is called the similarity dimension of the ratio list and we assume throughout that $n-1<D<n$. (We note that $D<n$ is equivalent to the finiteness of the total volume of the spray.)

The goal of tube formulas is to find a closed formula for the volume of the inner $\varepsilon$-tube of $\cup G_k$, where surprisingly the complex dimensions will take the stage. (By the inner $\varepsilon$-tube of an open set it is understood the set of points of this set with distance to the boundary less than $\varepsilon$.)

We will consider below a monophase generator $G$, though the pluriphase generators could be handled along the same lines. The adjective ``monophase'' means that the volume $V_G(\varepsilon)$ of the inner $\varepsilon$-tube is given by a polynomial till the inradius $g$ of $G$:

\begin{equation} \label{ictubevolume}
V_G(\varepsilon)=\left\{
\begin{tabular}{ccl}
$ \underset{i=0}{\overset{n-1}{\sum}}\kappa_i \, \varepsilon^{n-i}$& for &$  \varepsilon< g$ \\
&&\\
$ {\rm Vol}(G)$ & for & $\varepsilon \geq g$
\end{tabular}%
\right.
\end{equation}

\begin{figure}[t]
\centering
\includegraphics[scale=0.72]{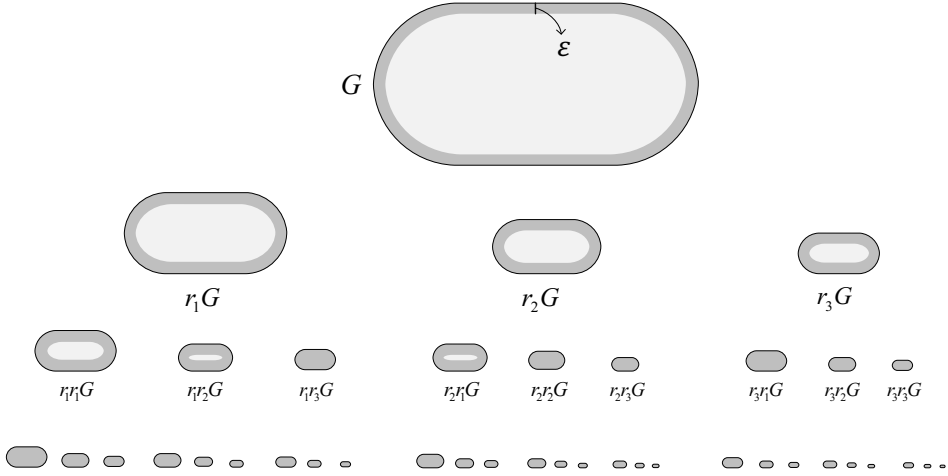}\caption{A spray $(G_k)_{k\in \mathbb{N}}$ and its inner $\varepsilon$-tube}
\label{figurespraytube}
\end{figure}

\section{A Functional Equation}\label{functionalequation}

It is almost obvious that the volume $V_{\cup G_k}(\varepsilon)$ of the inner $\varepsilon$-tube of $\cup G_k$ satisfies the following functional equation:
\begin{equation} \label{ictubevolumespray}
V_{\cup G_k} (\varepsilon)=
  r_1^n  \, V_{\cup G_k} \left( \frac{\varepsilon}{r_1} \right)  +  \cdots +   r_J^n  \, V_{\cup G_k} \left( \frac{\varepsilon}{r_J} \right) +V_{G}(\varepsilon)   \  \ \mbox{ for }  \varepsilon >0.
\end{equation}

\noindent For $\varepsilon \geq g$, $V_{\cup G_k} (\varepsilon)$ can instantly be computed to yield
\[\frac{{\rm Vol}(G)}{1-(r_1^n+r_2^n+\cdots+r_J^n)}.\]

The main idea of the present note is to use this functional equation to obtain the Mellin transform of $V_{\cup G_k}(\varepsilon)$ from which one can get back the tube volume as a sum of residues of an appropriate meromorphic function.

We recall that the Mellin transform of a function $f:(0,\infty) \rightarrow \mathbb{R}$ is given by
\[\widetilde f (s)=\int_0^{\infty} f(x) \, x^{s-1} dx.\]

If this integral exists for some $c\in \mathbb{R}$ and if the function $f$ is continuous at $x_0 \in (0,\infty)$ and of bounded variation in a neighborhood of $x_0$, then $f(x_0)$ can be recovered by the inverse Mellin transform (\cite{Titch})
\[\frac{1}{2\pi \, \mathbbm{i}} \lim_{T \to \infty} \int_{c-\mathbbm{i} T}^{c+\mathbbm{i} T} \widetilde f(s) \, x^{-s} ds.\]

In our case $V_{\cup G_k}$ is continuous and $O(1)$ as $\varepsilon \to \infty$. To apply the Mellin transform we need an estimate as $\varepsilon \to 0$.
\begin{lemma}\label{renewalden}
  $V_{\cup G_k}(\varepsilon)=O\left( \varepsilon^{n-D}  \right)$ as $\varepsilon \to 0$.
\end{lemma}
\begin{proof}
  See, for example, \cite{LeVas}.
\end{proof}

Before applying the Mellin transform to the above functional equation, it will be technically more advantageous to use a normalized version of the tube volume, so we define $\displaystyle f(\varepsilon):=\frac{V_{\cup G_k}(\varepsilon)}{\varepsilon^n}$. This function will simplify the above functional equation rendering it to the following one:
\begin{equation} \label{ictubevolumespraynormed}
f(\varepsilon) = f\left( \frac{\varepsilon}{r_1} \right)  +  \cdots +   f\left( \frac{\varepsilon}{r_J} \right) +\frac{V_{G}(\varepsilon)}{\varepsilon^n}   \ \  \mbox{ for }   \varepsilon >0.
\end{equation}

The function $f$ is continuous, of locally bounded variation (since $V_{\cup G_k}(\varepsilon)$ is locally polynomial for a monophase $G$), $f(\varepsilon)=O\left( \varepsilon^{-n}  \right)$ as $\varepsilon \to \infty$ and by Lemma~\ref{renewalden} $f(\varepsilon)=O\left( \varepsilon^{-D}  \right)$ as $\varepsilon \to 0$. So, the integral
\[\int_0^{\infty} f(\varepsilon) \, \varepsilon^{s-1} d\varepsilon\]
exists for any  $s$ with $D<{\rm Re}(s)<n$.

It can readily be computed that for $n-1<{\rm Re}(s)<n$,
\[\int_0^{\infty} V_G (\varepsilon) \, \varepsilon^{s-n-1} d\varepsilon =\sum \limits_{i=0}^{n} \kappa_i \frac{g^{s-i}}{s-i},\]
where ${\rm Vol}(G)$ is denoted by $-\kappa_n$.

Applying the Mellin transform to both sides of (\ref{ictubevolumespraynormed}) we get for $D<{\rm Re}(s)<n$,
\begin{eqnarray*}
\widetilde f(s)&=& r_1^s \widetilde f(s) + r_2^s \widetilde f(s)+\cdots + r_J^s \widetilde f(s) +  \int_0^{\infty} V_G (\varepsilon) \, \varepsilon^{s-n-1} d\varepsilon \, , \\
\widetilde f(s)&=& \frac{1}{1-\sum \limits_{j=1}^{J} r_j^s} \sum \limits_{i=0}^{n} \kappa_i \frac{g^{s-i}}{s-i}.
\end{eqnarray*}

Taking the inverse Mellin transform, we get  (for $D < c <n$) \[f(\varepsilon)= \frac{1}{2\pi \mathbbm{i}} \lim_{T \to \infty} \int_{c-\mathbbm{i}T}^{c+\mathbbm{i}T} \frac{\sum \limits_{i=0}^{n} \kappa_i \frac{g^{s-i}}{s-i}}{1-\sum \limits_{j=1}^{J} r_j^s} \varepsilon^{-s} ds \] and a well-known procedure using the residue theorem (explained in \cite{LaFra})  yields the tube formula
\[f(\varepsilon)=\sum_{\omega \in \mathfrak{D} \cup \{0,1,\dots,n-1 \}} {\rm res } \left (  \frac{\ \sum \limits_{i=0}^{n} \kappa_i \frac{g^{s-i}}{s-i}\ }{\ 1-\sum \limits_{j=1}^{J} r_j^s} \ \varepsilon^{-s}\, ; \, \omega      \right ) \ \ \ (\mbox{for } \varepsilon<g),\]
where $\mathfrak{D}$  is the set of zeros of the function $1-\sum \limits_{j=1}^{J} r_j^s$ (introduced and called complex dimensions by Lapidus); i.e. \[V_{\cup G_k} (\varepsilon) = \sum_{\omega \in \mathfrak{D} \cup \{0,1,\dots,n-1 \}} {\rm res}  \left ( \varepsilon^{n-s} \ \frac{\ \sum \limits_{i=0}^{n} \kappa_i \frac{g^{s-i}}{s-i}\ }{\ 1-\sum \limits_{j=1}^{J} r_j^s\ }  \,; \, \omega      \right ) \ \ \ (\mbox{for } \varepsilon<g).\]

\end{document}